\documentclass[11pt,reqno]{amsart}

\usepackage{amsmath,amssymb,amsthm}
\usepackage{bm}
\usepackage{natbib}
\usepackage{graphicx}
\usepackage{multirow}
\usepackage{here}
\usepackage{fullpage}
\usepackage{bm}
\usepackage{url}
\usepackage{color}
\usepackage{mathrsfs}

\makeatletter

\@addtoreset{equation}{section}
\makeatletter

\newtheorem{theorem}{Theorem}[section]

\newtheorem{lemma}{Lemma}[section]
\newtheorem{proposition}{Proposition}[section]
\newtheorem{assumption}{Assumption}[section]
\theoremstyle{definition}
\newtheorem{definition}{Definition}[section]
\newtheorem{remark}{Remark}[section]

\renewcommand{\tilde}{\widetilde}
\renewcommand{\hat}{\widehat}

\DeclareMathOperator{\Var}{Var}
\DeclareMathOperator{\Cov}{Cov}

\begin{document}

\title[]{Nonparametric regression for locally stationary functional time series}
\thanks{} 

\author[D. Kurisu]{Daisuke Kurisu}

%\date{First version: April 17, 2021. This version: \today}

\address[D. Kurisu]{Graduate School of International Social Sciences, Yokohama National University\\
79-4 Tokiwadai, Hodogaya-ku, Yokohama 240-8501, Japan.
}
\email{kurisu-daisuke-jr@ynu.ac.jp}

\begin{abstract}
In this study, we develop an asymptotic theory of nonparametric regression for a locally stationary functional time series. First, we introduce the notion of a locally stationary functional time series (LSFTS) that takes values in a semi-metric space. Then, we propose a nonparametric model for LSFTS with a regression function that changes smoothly over time.  We establish the uniform convergence rates of a class of kernel estimators, the Nadaraya-Watson (NW) estimator of the regression function, and a central limit theorem of the NW estimator. 
\end{abstract}


\maketitle

\section{Introduction}

In an increasing number of situations, the collected data appear as functional or curve time series coming from different research fields such as biometrics (\cite{ChMu09}), environmetrics (\cite{AuDuNo15}), econometrics (\cite{BuHaHoNe09}, \cite{BuHo18}), and finance (\cite{KoZh12}, \cite{ChLeTu16}, \cite{LiRoSh20}). We refer to \cite{FeVi06} and \cite{HoKo12} as standard references for functional time series analysis. 

In the literature on functional time series analysis, most studies are based on (linear) stationary models (e.g., \cite{Bo00, Bo02}, \cite{DeSh05}, \cite{AnPaSa06}, \cite{AuDuNo15}). However, many functional time series exhibit a nonstationary behavior. For example, in the financial industry, implied volatility of an option as a function of moneyness changes over time. We can also find other examples of possibly nonstationary functional time series in \cite{vaEi18}. One way to model nonstationary behavior is provided by the theory of locally stationary processes.

Locally stationary processes, as proposed by \cite{Da97}, are nonstationary time series that allow parameters of the time series to be time-dependent. They can be approximated by a stationary time series locally in time, which enables asymptotic theories to be established for the estimation of time-dependent characteristics. In time series analysis, locally stationary models are mainly considered in a parametric framework with time-varying coefficients. For example, we refer to \cite{DaSu06}, \cite{FrSaSu08}, \cite{KoLi12}, \cite{Zh14} and \cite{Tr17}. Moreover, nonparametric methods for stationary and nonstationary time series models have also been developed. We refer to, among others, \cite{Ma96}, \cite{FaYa03} and \cite{Ha08} for stationary time series as well as \cite{Kr09}, \cite{Vo12}, \cite{ZhWu15}, and \cite{Tr19} for contributions on locally stationary time series. Recently, the notion of local stationarity has been extended to spatial data by \cite{MaYa18}, \cite{Pe18}, and \cite{Ku21}. We refer to \cite{DaRiWu19} for a general theory in the literature on locally stationary processes. 

In contrast with the abovementioned studies regarding (nonparametric) time series analysis, studies pertaining to nonparametric methods for locally stationary functional time series are scarce despite empirical interest in modeling time-varying dependence structures of functional data. We refer to the works of \cite{vaEi18} and \cite{Auva20} as important recent contributions on the theoretical development for statistical methods for locally stationary functional time series. We also mention \cite{Ku21b} who investigates functional principal component analysis for locally stationary functional data. 

The objective of this study is to develop a framework of nonparametric regression for a locally stationary functional time series that takes values in a semi-metric space (e.g., Banach and Hilbert spaces). In particular, we first propose a nonparametric regression model for a locally stationary functional time series with a regression function that is allowed to change smoothly over time. Then, we (1) derive the uniform convergence rate and (2) establish the point-wise asymptotic normality of a kernel estimator for the regression function. To attain the first objective, we derive uniform convergence rates for a class of estimators based on kernel averages, which are crucial for demonstrating our main results. As these estimators include a wide range of kernel-based estimators such as the Nadaraya-Watson (NW) estimators, the general results are of independent interest. Our results extend the results of \cite{Vo12}, who studies univariate locally stationary time series, and of \cite{Ma05} and \cite{FeVi06}, who study stationary functional time series to our framework. To the best of our knowledge, this is the first study to develop an asymptotic theory of nonparametric regression for locally stationary functional time series. 

The organization of this paper is as follows. In Section \ref{Sec: setting}, we introduce the notion of local stationarity for functional time series that takes values in a semi-metric space and explain the dependence structure of the functional time series. In Section \ref{Sec: Main}, we present the main results including the uniform convergence rates of kernel estimators and asymptotic normality of the NW estimator of the regression function. All proofs are included in the Appendix. 

\subsection{Notations}
For any positive sequences $a_{n}$ and $b_{n}$, we write $a_{n} \lesssim b_{n}$ if a constant $C >0$ independent of $n$ exists such that $a_{n} \leq Cb_{n}$ for all $n$,  $a_{n} \sim b_{n}$ if $a_{n} \lesssim b_{n}$ and $b_{n} \lesssim a_{n}$, and $a_{n} \ll b_{n}$ if $a_{n}/b_{n} \to 0$ as $n \to \infty$. For any $a,b \in \mathbb{R}$, we write $a \vee b=\max\{a,b\}$ and $a \wedge b=\min\{a,b\}$. We use the notation $\stackrel{d}{\to}$ to denote convergence in distribution. For $x \in \mathbb{R}$, $\lfloor x \rfloor$ denotes the integer part $x$.

\section{Settings}\label{Sec: setting}

In this section, we introduce the notion of a locally stationary functional time series that extends the notion of local stationarity introduced by \cite{Da97}. Furthermore, we discuss dependence structures of the functional time series.

\subsection{Model}

Let $\{Y_{t,T}, X_{t,T}\}_{t=1}^{T}$ be random variables where $Y_{t,T}$ is real-valued and $X_{t,T}$ takes values in some semi-metric space $\mathscr{H}$ with a semi-metric $d(\cdot,\cdot)$. In most applications, $\mathscr{H}$ is a Banach or Hilbert space with norm $\|\cdot\|$ so that $d(u,v)=\|u-v\|$. 

In this study, we consider the following model: 
\begin{align}\label{def: model}
Y_{t,T} &= m\left({t \over T}, X_{t,T}\right) + \sigma\left({t \over T}, X_{t,T}\right)\varepsilon_{t},\ t=1,\dots, T, 
\end{align}
where $\{\varepsilon_{t}\}_{t \in \mathbb{Z}}$ is a sequence of independent and identically distributed random variables that is independent of $\{X_{t,T}\}_{t=1}^{T}$. We write $\sigma\left({t \over T}, X_{t,T}\right)\varepsilon_{t}$ as $\varepsilon_{t,T}$ for notational convenience. We also assume that $\{X_{t,T}\}$ is a locally stationary functional time series, and the regression function $m$ is allowed to change smoothly over time. 

\subsection{Local stationarity}
Intuitively, a functional time series, $\{X_{t,T}\}_{t=1}^{T}$ ($T \to \infty$), is locally stationary if it behaves approximately stationary in local time. We refer to \cite{DaSu06} and \cite{DaRiWu19} for the idea of a locally stationary time series and its general theory, as well as to \cite{vaEi18} and \cite{Auva20} for the notion of local stationarity for a Hilbert space-valued time series. To ensure that it is locally stationary around each rescaled time point $u$, a process $\{X_{t,T}\}$ can be approximated by a stationary functional time series $\{X_{t}^{(u)}\}$ stochastically. This concept can be defined as follows. 

\begin{definition}\label{def: LSfTS}
The $\mathscr{H}$-valued stochastic process $\{X_{t,T}\}_{t=1}^{T}$ is locally stationary if for each rescaled time point $u \in [0,1]$, there exists an associated $\mathscr{H}$-valued process $\{X_{t}^{(u)}\}_{t \in \mathbb{Z}}$ with the following properties: 
\begin{itemize}
\item[(i)] $\{X_{t}^{(u)}\}_{t \in \mathbb{Z}}$ is strictly stationary.
\item[(ii)] It holds that 
\begin{align}\label{def: LSfTS-ineq}
d\left(X_{t,T}, X_{t}^{(u)}\right) \leq \left(\left|{t \over T} - u\right| + {1 \over T}\right)U_{t,T}^{(u)}\ a.s.,
\end{align}
for all $1 \leq t \leq T$, where $\{U_{t,T}^{(u)}\}$ is a process of positive variables satisfying $E[(U_{t,T}^{(u)})^{\rho}]<C$ for some $\rho>0$ and $C<\infty$ that are independent of $u, t$, and $T$.
\end{itemize}
\end{definition}

Definition \ref{def: LSfTS} is a natural extension of the notion of local stationarity for the real-valued time series introduced by \cite{Da97}.  Moreover, our definition corresponds to that of \cite{vaEi18} (Definition 2.1) when $\mathscr{H}$ is the Hilbert space $L_{\mathbb{R}}^{2}([0,1])$ of all real-valued functions that are square integrable with respect to the Lebesgue measure on the unit interval $[0,1]$ with the $L_{2}$-norm given by
\[
\|f\|_{2} = \sqrt{\langle f,f \rangle},\ \langle f,g \rangle = \int_{0}^{1}f(t)g(t)dt,
\] %in Section \ref{Ex-LSRF}. 
where $f,g\in L_{\mathbb{R}}^{2}([0,1])$. The authors also give sufficient conditions so that an $L_{\mathbb{R}}^{2}([0,1])$-valued stochastic process $\{X_{t,T}\}$ satisfies (\ref{def: LSfTS-ineq}) with $d(f,g) = \|f-g\|_{2}$ and $\rho=2$. 

\subsection{Mixing condition}
 Let $(\Omega, \mathcal{F}, P)$ be a probability space, and let $\mathcal{A}$ and $\mathcal{B}$ be subfields of $\mathcal{F}$. Define
 \begin{align*}
 \alpha(\mathcal{A},\mathcal{B}) = \sup_{A \in \mathcal{A}, B \in \mathcal{B}}|P(A \cap B)-P(A)P(B)|.
 \end{align*}
 Moreover, for an array $\{Z_{t,T}: 1\leq t\leq T\}$, define the coefficients
 \begin{align*}
 \alpha(k) = \sup_{t,T: 1\leq t\leq T-k}\alpha(\sigma(Z_{s,T}: 1\leq s\leq t), \sigma(Z_{s,T}:t+k\leq s \leq T)),
 \end{align*}
 where $\sigma(Z)$ is the $\sigma$-field generated by $Z$. The array $\{Z_{t,T}\}$ is said to be $\alpha$-mixing (or strongly mixing) if $\alpha(k) \to 0$ as $k \to \infty$. 
 
\section{Main results}\label{Sec: Main}

In this section, we consider general kernel estimators and derive their uniform convergence rates. Based on the result, we derive the uniform convergence rate and asymptotic normality of the NW estimator for the regression function in model (\ref{def: model}).

\subsection{Kernel estimation for regression functions}\label{section-kernel}

We consider the following kernel estimator for $m(u,x)$ in model (\ref{def: model}): 
\begin{align}\label{def: m-est}
\hat{m}(u,x) &= {\sum_{t=1}^{T}K_{1,h}(u-t/T)K_{2,h}(d(x, X_{t,T}))Y_{t,T} \over \sum_{t=1}^{T}K_{1,h}(u-t/T)K_{2,h}(d(x, X_{t,T}))},
\end{align}
where $K_{1}$ and $K_{2}$ denote one-dimensional kernel functions, and we used the notations $K_{j,h}(v) = K_{j}(v/h)$, $j=1,2$. Here, $h=h_{T}$ is a bandwidth satisfying $h \to 0$ as $T \to \infty$. 

Before we state the main results, we summarize the assumptions made for model (\ref{def: model}) and the kernel functions. These assumptions are standard, and similar assumptions are made by \cite{Ma05} and \cite{vaEi18}.
\begin{assumption}\label{Ass-M}%[Assumption M] 
\begin{enumerate}
\item[(M1)] The process $\{X_{t,T}\}$ is locally stationary, that is, $\{X_{t,T}\}$ satisfies Definition \ref{def: LSfTS}. 
\item[(M2)]  Let $B(x,h)= \{y \in \mathscr{H}: d(x,y)\leq h\}$ denote the ball of radius $h$ centered in $x \in \mathscr{H}$. We assume that there exist positive constants $c_{d}<C_{d}$, such that for all $u \in [0,1]$, all $x \in \mathscr{H}$, and all $h>0$,
\begin{align}\label{small-ball-ass}
0<c_{d}\phi(h)f_{1}(x)\leq P(X_{t}^{(u)} \in B(x,h))=:F_{u}(h;x)\leq C_{d}\phi(h)f_{1}(x),
\end{align}
where $\phi(h) \to 0$ as $h \to \infty$, and $f_{1}(x)$ is a nonnegative functional in $x \in \mathscr{H}$. Moreover, there exist constants $C_{\phi}>0$ and $\varepsilon_{0}>0$ such that for any $0<\varepsilon<\varepsilon_{0}$, 
\begin{align}\label{small-Ball-low}
\int_{0}^{\varepsilon}\phi(u)du>C_{\phi}\varepsilon \phi(\varepsilon).
\end{align}
\item[(M3)] $\sup_{s,t,T}\sup_{s \neq t}P((X_{s,T}, X_{t,T}) \in B(x,h)\times B(x,h))\leq \psi(h)f_{2}(x)$, where $\psi(h) \to 0$ as $h \to 0$, and $f_{2}(x)$ is a nonnegative functional in $x \in \mathscr{H}$. We assume that the ratio $\psi(h)/\phi^{2}(h)$ is bounded. 
\item[(M4)] $m(u, x)$ is twice continuously partially differentiable with respect to $u$. We also assume that
\begin{align*}
\sup_{u \in [0,1]}|m(u,x)-m(u,y)|&\leq c_{m}d(x,y)^{\beta}
\end{align*}
for all $x,y \in \mathscr{H}$ for some $c_{m}>0$ and $\beta>0$.
\end{enumerate}
\end{assumption}

Condition (\ref{small-Ball-low}) is satisfied for fractal-type processes (i.e., $\phi(h) \sim \varepsilon^{\tau}$ as $\varepsilon \to 0$ for some $\tau>0$). %and a class of exponential-type processes (i.e., $\phi(\varepsilon) \sim e^{-{1 \over \varepsilon^{\tau_1}}\left(\log({1 \over \varepsilon})\right)^{\tau_{2}}}$ as $\varepsilon \to 0$ for some $\tau_{1}>0$ and $\tau_{2} \in \mathbb{R}$). 
In particular the condition (\ref{small-ball-ass}) is consistent with the assumption made by \cite{GaHaPr98}. If the space $\mathscr{H}$ is a separable Hilbert space, it is possible to choose a semi-metric for which condition (\ref{small-Ball-low}) is fulfilled with $\phi(\varepsilon) \sim \varepsilon^{\tau_{0}}$ as $\varepsilon \to 0$ for some $\tau_{0}>0$ (Lemma 13.6 in \cite{FeVi06}). %Let $C([0,1])$ be the space of real-valued continuous functions with metric $d$ associated with the supremum norm $\|x\|=\sup_{t\in[0,1]}|x(t)|$, $x \in C([0,1])$. It is known that some classes of diffusion and Gaussian processes are exponential-type processes that take values in $C([0,1])$. 
We refer to \cite{Bo99} and \cite{FeVi06} %, and \cite{LiSh01} 
for detailed discussions on fractal-type %- and exponential-type 
processes and the effect of a semi-metric on the small ball probability $F_{u}(h;x)$. Condition (M3) gives the behavior of the joint distribution of $(d(X_{s,T},x), d(X_{t,T},x))$ near the origin. Assumptions (M1), (M2) and (M3) can be satisfied by a class of random coefficient models: For each $u \in [0,1]$, let $\{H(u,t) = (H_1(u,t), \dots, H_d(u,t))'\}_{t \in \mathbb{Z}}$ be a $d$-variate ($\alpha$-mixing) stationary time series with independent components. Assume that there exist random variables $\bar{H}_k(u,t,T)$, $k=1,\dots,d$ such that 
\begin{align}\label{LS-d}
|H_k(t/T, t) - H_k(u,t)| &\leq \left(\left|{t \over T} - u\right| + {1 \over T}\right)\bar{H}_{k}(u,t,T),\ E[\bar{H}_k(u,t,T)^2] \leq C<\infty
\end{align} 
for some positive constant $C$ independent of $u,t$, and $T$. Note that we can construct locally stationary time series $H_k(t/T,t)$, $k=1,\dots, d$ that satisfy (\ref{LS-d}) (see \cite{Vo12} for example). 
Let $\{b_k(s)\}_{k=1}^{\infty}$ be an orthogonal basis of $L_{\mathbb{R}}^2([0,1])$. Define $X_t^{(u)}(s) = \sum_{k=1}^{d}H_k(u,t)b_k(s)$, $X_{t,T}(s) = X_t^{(t/T)}(s)$. Then we have
\begin{align*}
\|X_{t,T} - X_t^{(u)}\| &\!=\! \left(\sum_{k=1}^{d}(H_k(t/T,t) - H_k(u,t))^2\right)^{1/2}\!\!\!\! \leq \left(\left|{t \over T} - u\right| + {1 \over T}\right)\left(\sum_{k=1}^{d}\bar{H}_k^2(u,t,T)\!\right)^{1/2}.
\end{align*}
This implies that $X_{t,T}$ is a locally stationary functional time series. Further, the distributions of the functional time series $X_t^{(u)}$ and $X_{t,T}$ are completely determined by the multivariate time series $H(u,t)$ and $H(t/T,t)$, respectively. In this case, the functions $\phi(h)$ and $\psi(h)$ in Assumptions (M2) and (M3) can be $\phi(h) \sim h^d$ and $\psi(h) \sim h^{2d}$.   Condition (M4) is concerned with the smoothness and continuity of the regression function $m(u,x)$ with respect to $u$ and $x$, respectively. Conditions (M3) and (M4) are consistent with the assumptions made by \cite{Ma05} and \cite{FeVi06}.

We assume the following conditions on $\sigma$ and kernel functions. Similar assumptions are made by \cite{Ma05} and \cite{FeVi06}, and \cite{Vo12}.

\begin{assumption}\label{Ass-Sig}
\begin{enumerate}
\item[($\Sigma$1)] $\sigma:[0,1]\times \mathscr{H} \to \mathbb{R}$ is bounded by some constant $C_{\sigma}<\infty$ from above and by some constant $c_{\sigma}>0$ from below, that is, $0<c_{\sigma}\leq\sigma(u,x)\leq C_{\sigma}<\infty$ for all $u$ and $x$. 
\item[($\Sigma$2)] $\sigma$ is Lipschitz continuous with respect to $u$
\item[($\Sigma$3)] $\sup_{u \in [0,1]}\sup_{y:d(x,y)\leq h}|\sigma(u,x)-\sigma(u,y)|=o(1)$ as $h \to 0$.
\end{enumerate}
\end{assumption}

\begin{assumption}\label{Ass-KB}%[Assumption KB]
\begin{enumerate}
\item[(KB1)] The kernel $K_{1}$ is symmetric around zero, bounded, and has a compact support, that is, $K_{1}(v) = 0$ for all $|v|>C_{1}$ for some $C_{1}<\infty$. Moreover, $\int K_{1}(z)dz=1$ and $K_{1}$ is Lipschitz continuous, that is, $|K_{1}(v_{1}) - K_{1}(v_{2})| \leq C_{2}|v_{1} - v_{2}|$ for some $C_{2}<\infty$ and all $v_{1},v_{2} \in \mathbb{R}$. 
\item[(KB2)] The kernel $K_{2}$ is nonnegative, bounded, and has support in $[0,1]$ such that $0<K_{2}(0)$ and $K_{2}(1)=0$. Moreover, $K'_{2}(v)=dK_{2}(v)/dv$ exists on $[0,1]$ and satisfies $C'_{1}\leq K'_{2}(v) \leq C'_{2}$ for two real constants $-\infty<C'_{1}<C'_{2}<0$.
\end{enumerate}
\end{assumption}
Conditions ($\Sigma$1) and ($\Sigma$2) are consistent with the assumption made by \cite{Vo12}. Condition ($\Sigma$3) is required for investigating the asymptotic property of the variance of $\hat{m}(u,x)$ to establish its asymptotic normality.
The assumption on kernel functions $K_{1}$ and $K_{2}$ are standard in the literature and satisfied by popular kernel functions, such as the (asymmetric) triangle and quadratic kernels. 

\subsection{Uniform convergence rates for kernel estimators}

As a first step to study the asymptotic properties of the estimator (\ref{def: m-est}), we analyze the following general kernel estimator: 
\begin{align}\label{def: genKernel-est}
\hat{\psi}(u,x) &= {1 \over Th\phi(h)}\sum_{t=1}^{T}K_{1,h}\left(u - {t \over T}\right)K_{2,h}\left(d(x, X_{t,T})\right)W_{t,T},
\end{align}
where $\{W_{t,T}\}$ is an array of one-dimensional random variables. Some kernel estimators, such as NW estimators, can be represented by (\ref{def: genKernel-est}). In this study, we use the results with $W_{t,T}=1$ and $W_{t,T} = \varepsilon_{t,T}$. 

Next, we derive the uniform convergence rate of $\hat{\psi}(u, x) - E[\hat{\psi}(u,x)]$. We assume the following for the components in (\ref{def: genKernel-est}). 

\begin{assumption}\label{Ass-E}%[Assumption E]
\begin{enumerate}
\item[(E1)] It holds that $\sup_{t,T}\sup_{x \in \mathscr{H}}E[|W_{t,T}|^{\zeta}|X_{t,T}=x] \leq C$ for some $\zeta > 2$ and $C <\infty$. 
\item[(E2)] The $\alpha$-mixing coefficients of the array $\{X_{t,T}, W_{t,T}\}$ satisfy $\alpha(k) \leq Ak^{-\gamma}$ for some $A>0$ and $\gamma> 2$. We also assume that $\delta+1<\gamma(1-{2 \over \nu})$ for soe $\nu> 2$ and $\delta>1-{2 \over \nu}$, and 
\begin{align}\label{mixing-CLT}
h^{2(1\wedge \beta)-1}\left(\phi(h)\lambda_{T} + \sum_{k=\lambda_{T}}^{\infty}k^{\delta}(\alpha(k))^{1-{2 \over \nu}}\right) \to 0,
\end{align}
as $T \to \infty$, where $\lambda_{T}=[(\phi(h))^{-(1-{2 \over \nu})/\delta}]$. 
\end{enumerate}
\end{assumption}

We need (\ref{mixing-CLT}) to establish the asymptotic normality of $\hat{m}(u,x)$. We also need Condition (\ref{mixing-CLT}) to show the asymptotic negligibility of the bias of $\hat{m}(u,x)$. Condition (E2) is consistent with assumptions made in \cite{Ma05} and \cite{Vo12}.

Furthermore, we assume the following regularity conditions on $h$ and $\phi(h)$. 
\begin{assumption}\label{Ass-R}%[Assumption R]
As $T \to \infty$, 
\begin{enumerate}
\item[(R1)] ${(\log T)^{{\gamma \over 2}+\zeta_{0}(\gamma +1)} \over T^{{\gamma \over 2}-1-{\gamma + 1 \over \zeta}}h^{{\gamma \over 2}+1}\phi(h)^{{\gamma \over 2}}} \to 0$ for some $\zeta_{0}>0$, and 
\item[(R2)] $Th^{3}, Th\phi(h) \to \infty$,
\end{enumerate}
where $\zeta$ and $\gamma$ are positive constants that appear in Assumption \ref{Ass-E}.   
\end{assumption}

Condition (R1) is required to apply an exponential inequality for $\alpha$-mixing sequence to establish the uniform convergence rate of the general estimator and $\hat{m}(u,x)$. Condition (R2) is concerned with the bias and the convergence rate of the general estimator $\hat{\psi}(u,x)$.

%\begin{assumption}\label{Ass-small-Ball}
%\begin{enumerate}
%\item[(SB1)] There exist constants $C_{\phi}>0$ and $\varepsilon_{0}>0$ such that for any $0<\varepsilon<\varepsilon_{0}$, $\int_{0}^{\varepsilon}\phi(u)du>C_{\phi}\varepsilon \phi(\varepsilon)$.
%\end{enumerate}
%\end{assumption}

The next theorem generalizes the uniform convergence results of \cite{Vo12} to a functional time series.

\begin{proposition}\label{Prop: unif-rate}
Assume that Assumptions \ref{Ass-M} (M1), (M2), \ref{Ass-KB}, \ref{Ass-E}, and \ref{Ass-R} are satisfied. Then the following result holds for any $x \in \mathscr{H}$: 
\begin{align*}
\sup_{u \in [0,1]}|\hat{\psi}(u,x) - E[\hat{\psi}(u,x)]| &= O_{p}\left(\sqrt{{\log T  \over Th\phi(h)}}\right).
\end{align*}
\end{proposition}
Apart from $h$, which comes from the smoothing in time direction, the convergence rate in the above proposition is the same as the point-wise convergence rate of the (nonparametric) regression function obtained in \cite{FeVi06} for a strictly stationary functional time series.
The next theorem provides the uniform convergence rate of the kernel estimator $\hat{m}(u,x)$. 
\begin{theorem}\label{Thm: unif-rate}
Assume that Assumptions \ref{Ass-M}, \ref{Ass-Sig}, \ref{Ass-KB}, and \ref{Ass-R} are satisfied and that Assumption \ref{Ass-E} is satisfied with $W_{1,T}=1$ and $W_{t,T}=\varepsilon_{t,T}$. Then, the following result holds for any $x \in \mathscr{H}$:
\begin{align}\label{unif-rate-m}
\sup_{u \in [C_{1}h,1-C_{1}h]}|\hat{m}(u,x) - m(u,x)|&= O_{p}\left(\sqrt{\log T \over Th\phi(h)} + h^{2\wedge\beta}\right). 
\end{align}
\end{theorem}

Theorem \ref{Thm: unif-rate} generalizes the results on point-wise convergence in \cite{FeVi06} and the results in \cite{Ma05} for a strictly stationary functional time series case to our setting. Using Proposition \ref{Prop: unif-rate} the stochastic part is shown to be of order $O_{p}(\sqrt{\log T / Th\phi(h)})$. Compared with Theorem 4.2 in \cite{Vo12}, we do not have the bias term that comes from the approximation error of $X_{t,T}$ by $X_{t}^{(u)}$. Indeed, under our assumptions, the approximation error is $O\left(T^{-1}h^{(1 \wedge \beta)-1}\phi^{-1}(h)\right) \ll h^{2 \wedge \beta}$.
\begin{remark}
For a fractal-type process $\{X_{t}^{(u)}\}$, the right-hand side of (\ref{unif-rate-m}) with $\beta \leq 2$ is optimized by choosing $h \sim \left({\log T \over T}\right)^{{1 \over 2\beta+\tau+1}}$, and the optimized rate is 
\[
\sup_{u \in [C_{1}h,1-C_{1}h]}|\hat{m}(u,x) - m(u,x)|= O_{p}\left( \left({\log T \over T}\right)^{{\beta \over  2\beta+\tau+1}}\right). 
\]
%If $\beta \leq 2$ and $\{X_{t}^{(u)}\}$ is an exponential-type process, then by choosing $h \sim \left(\log T \right)^{-\kappa}$ for some $\kappa>0$, the uniform rate is 
%\[
%\sup_{u \in [C_{1}h,1-C_{1}h]}|\hat{m}(u,x) - m(u,x)|= O_{p}\left( (\log T)^{-\kappa_{1}}\right)\ \text{for some $\kappa_{1}>0$}. 
%\]
\end{remark}

\subsection{Asymptotic normality for kernel estimators}
In this section, we provide a central limit theorem for the kernel estimator $\hat{m}(u,x)$. To establish the asymptotic normality of the NW estimator $\hat{m}(u,x)$, we additionally make the following assumption, which is used to employ Bernstein's big-block and small-block procedure.   

\begin{assumption}\label{Ass-block}
There exists a sequence of positive integers $\{v_{T}\}$ satisfying $v_{T}\to \infty$, $v_{T}=o(\sqrt{Th\phi(h)})$ and $\sqrt{{T \over h\phi(h)}}\alpha(v_{T}) \to 0$ as $T \to \infty$. 
\end{assumption}

Observe that 
\begin{align*}
\hat{m}(u, x) - m(u,x) &= {1 \over \hat{m}_{1}(u, x)}\left(\hat{g}_{1}(u,x) + \hat{g}_{2}(u,x) - m(u,x)\hat{m}_{1}(u,x)\right)\\
&= {1 \over \hat{m}_{1}(u, x)}\left(\hat{g}_{1}(u,x) + \hat{g}^{B}(u,x)\right),
\end{align*}
where
\begin{align*}
\hat{m}_{1}(u,x) &= {1 \over Th\phi(h)}\sum_{t=1}^{T}K_{1,h}\left(u - {t \over T}\right)K_{2,h}\left(d\left(x, X_{t,T}\right)\right),\\
\hat{g}_{1}(u,x) &= {1 \over Th\phi(h)}\sum_{t=1}^{T}K_{1,h}\left(u - {t \over T}\right)K_{2,h}\left(d\left(x, X_{t,T}\right)\right)\varepsilon_{t, T},\\
\hat{g}_{2}(u,x) &= {1 \over Th\phi(h)}\sum_{t=1}^{T}K_{1,h}\left(u - {t \over T}\right)K_{2,h}\left(d\left(x, X_{t,T}\right)\right)m\left({t \over T} , X_{t,T}\right). 
\end{align*}
Under the same assumption in Theorem \ref{Thm: unif-rate}, we can show that $\Var(\hat{g}^{B}(u,x))=o\left({1 \over Th\phi(h)}\right)$ and $1/\hat{m}_{1}(u,x)=O_{p}(1)$. See the proof of Theorem \ref{Thm: normality} for details. Then, we have
\begin{align*}
\hat{m}(u, x) - m(u,x) &= {\hat{g}_{1}(u,x) \over \hat{m}_{1}(u, x)} + B_{T}(u,x) + o_{p}\left(\sqrt{{1 \over Th\phi(h)}}\right),
\end{align*}
where $B_{T}(u,x)=E[\hat{g}^{B}(u,x)]/E[\hat{m}_{1}(u,x)]$ is the ``bias'' term and ${\hat{g}_{1}(u,x) \over \hat{m}_{1}(u, x)}$ is the ``variance'' term. 

In the following result, we set $K_{2}$ as the asymmetrical triangle kernel, that is, $K_{2}(x) = (1-x)I(x \in [0,1])$ to simplify the proof. 
\begin{theorem}\label{Thm: normality}
Assume that Assumptions  \ref{Ass-M}, \ref{Ass-Sig}, \ref{Ass-KB}, \ref{Ass-R}, and \ref{Ass-block} are satisfied and that Assumption \ref{Ass-E} is satisfied for both $W_{1,T}=1$ and $W_{t,T}=\varepsilon_{t,T}$. Then as $T\to \infty$, the following result holds for any $x \in \mathscr{H}$:
\begin{align*}
\sqrt{Th\phi(h)}(\hat{m}(u,x) - m(u,x)-B_{T}(u,x))&\stackrel{d}{\to} N(0, V(u,x)),
\end{align*}
where $B_{T}(u,x) = O(h^{2\wedge \beta})$ and 
\[
V(u,x) = \lim_{T \to \infty}Th\phi(h){\Var\left(\hat{g}_{1}(u,x)\right) \over E[\hat{m}_{1}(u,x)]}>0.
\]
\end{theorem}
Theorem \ref{Thm: normality} is an extension of the results in \cite{Ma05} and \cite{Vo12} to a locally stationary functional time series.  
In particular, the bias and variance expressions $B_{T}(u,x)$ and $V(u,x)$ are very similar to those in \cite{Ma05}. By requiring that $Th^{1+2(2 \wedge \beta)}\phi(h) \to 0$, the  bias $B_{T}(u,x)$ term is asymptotically negligible.
\section{Concluding remarks}

In this paper, we have developed an asymptotic theory for nonparametric regression models with time-varying regression function with locally stationary functional covariate. In particular, we derived uniform convergence rates of general kernel estimators and the NW estimator of the regression function. We also established a central limit theory of the NW estimator.

As discussed in \cite{Vo12}, it would be possible to provide the uniform convergence rate of $\hat{m}(u,x)$ over $(1-C_{1}h, 1] \times \{x\}$, which is important for forecasting purposes by using
boundary-corrected kernels or one-sided kernels. In both cases, we have to ensure that the kernels are compactly supported and they are
Lipschitz continuous to get the theory to work.

\section*{Acknowledgements}

D. Kurisu is partially supported by JSPS KAKENHI Grant Number 20K13468. I am grateful to Taisuke Otsu for his helpful comments. 

\newpage

\appendix

\section{Proofs}\label{Appendix: proof}

\subsection{Proofs for Section \ref{Sec: Main}}

\begin{proof}[Proof of Proposition \ref{Prop: unif-rate}]
Define $B = [0,1]$, $a_{T} = \sqrt{\log T/Th\phi(h)}$ and $\tau_{T} = \rho_{T}T^{1/\zeta}$ with $\rho_{T} = (\log T)^{\zeta_{0}}$ for some $\zeta_{0} >0$. Define 
\begin{align*}
\hat{\psi}_{1}(u,x) &= {1 \over Th\phi(h)}\sum_{t=1}^{T}K_{1,h}\left(u - {t \over T}\right)K_{2,h}\left(d\left(x, X_{t,T}\right)\right)W_{t,T}I(|W_{t,T}| \leq \tau_{T}),\\
\hat{\psi}_{2}(u,x) &= {1 \over Th\phi(h)}\sum_{t=1}^{T}K_{1,h}\left(u - {t \over T}\right)K_{2,h}\left(d\left(x, X_{t,T}\right)\right)W_{t,T}I(|W_{t,T}| > \tau_{T}).
\end{align*}
Note that $\hat{\psi}(u,x) - E[\hat{\psi}(u,x)] = \hat{\psi}_{1}(u,x) - E[\hat{\psi}_{1}(u,x)] + \hat{\psi}_{2}(u,x) - E[\hat{\psi}_{2}(u,x)]$.

(Step1) First we consider the term $\hat{\psi}_{2}(u,x) - E[\hat{\psi}_{2}(u,x)]$. 
\begin{align*}
P\left(\sup_{u \in B}|\hat{\psi}_{2}(u,x)|>a_{T}\right) &\leq P\left(|W_{t,T}| > \tau_{T}\ \text{for some $t=1,\hdots,T$}\right)\\
&\leq \tau_{T}^{-\zeta}\sum_{t=1}^{T}E[|W_{t,T}|^{\zeta}] \leq T\tau_{T}^{-\zeta} = \rho_{T}^{-\zeta} \to 0.
\end{align*}

\begin{align*}
E\left[|\hat{\psi}_{2}(u,x)|\right] &\leq {1 \over Th\phi(h)}\sum_{t=1}^{T}K_{1,h}\left(u - {t \over T}\right)E\left[K_{2,h}\left(d\left(x,X_{t,T}\right)\right)|W_{t,T}|I(|W_{t,T}|> \tau_{T})\right].
\end{align*}
Since
\begin{align*}
K_{2,h}\left(d\left(x,X_{t,T}\right)\right) &\leq \left|K_{2,h}\left(d\left(x,X_{t,T}\right)\right)-K_{2,h}\left(d\left(x,X_{t}^{(t/T)}\right)\right)\right| + K_{2,h}\left(d\left(x,X_{t}^{(t/T)}\right)\right)\\
&\leq h^{-1}\left|d\left(x,X_{t,T}\right)-d\left(x,X_{t}^{(t/T)}\right)\right| + K_{2,h}\left(d\left(x,X_{t}^{(t/T)}\right)\right)\\
&\leq h^{-1}d\left(X_{t,T}, X_{t}^{(t/T)}\right) + K_{2,h}\left(d\left(x,X_{t}^{(t/T)}\right)\right)\\
&\leq {1 \over Th}U_{t,T}^{(t/T)} + K_{2,h}\left(d\left(x,X_{t}^{(t/T)}\right)\right)
\end{align*}
and 
\begin{align*}
E\left[K_{2,h}\left(d\left(x,X_{t,T}\right)\right)|W_{t,T}|I(|W_{t,T}|> \tau_{T})\right] &\lesssim \tau_{T}^{-(\zeta-1)}E\left[K_{2,h}\left(d\left(x,X_{t,T}\right)\right)|W_{t,T}|^{\zeta}\right]\\
&\lesssim \tau_{T}^{-(\zeta-1)}E\left[K_{2,h}\left(d\left(x,X_{t,T}\right)\right)\right],
\end{align*}
we have 
\begin{align*}
&E\left[K_{2,h}\left(d\left(x,X_{t,T}\right)\right)|W_{t,T}|I(|W_{t,T}|> \tau_{T})\right]\\
&\quad \lesssim {1 \over Th\tau_{T}^{\zeta-1}}E[U_{t,T}^{(t/T)}] + \tau_{T}^{-\zeta+1}E\left[K_{2,h}\left(d\left(x,X_{t}^{(t/T)}\right)\right)\right]\\
&\quad \lesssim {1 \over Th\tau_{T}^{\zeta-1}} + \tau_{T}^{-\zeta+1}E[I(d(x,X_{t}^{(t/T)})\leq h)] \leq {1 \over Th\tau_{T}^{\zeta}} + \tau_{T}^{-\zeta}F_{t/T}(h;x) \lesssim \tau_{T}^{-\zeta+1}\phi(h).
\end{align*}
Then we have 
\begin{align*}
E\left[|\hat{\psi}_{2}(u,x)|\right] &\lesssim \tau_{T}^{-\zeta+1}\phi(h) {1 \over Th\phi(h)}\sum_{t=1}^{T}K_{1,h}\left(u - {t \over T}\right)\\
&\lesssim {1 \over \tau_{T}^{\zeta-1}}{1 \over Th}\sum_{t=1}^{T}K_{1,h}\left(u - {t \over T}\right) \lesssim {1 \over \tau_{T}^{\zeta-1}} = \rho_{T}^{-(\zeta-1)}T^{-{\zeta-1 \over \zeta}} \lesssim a_{T}. 
\end{align*}
For the third inequality, we used Lemma \ref{lem: g} below. 
As a result, $\sup_{u \in B}|\hat{\psi}_{2}(u,x) - E[\hat{\psi}_{2}(u,x)]| = O_{p}(a_{T})$. 

(Step2) Now we show $\sup_{u \in B}\left|\hat{\psi}_{1}(u,x) - E[\hat{\psi}_{1}(u,x)]\right| = O_{p}\left(a_{T}\right)$. Cover the region $B$ with $N \lesssim h^{-1}a_{T}^{-1}$ balls $B_{k,T} = \{u \in \mathbb{R}: |u - u_{k}| \leq a_{T}h\}$ and use $u_{k}$ to denote the mid point of $B_{k,T}$. In addition, let $K^{\ast}(w,v) = CI(|w| \leq 2C_{1})K_{2}(v)$ for $(w,v) \in \mathbb{R}^{2}$. Note that for $u \in B_{k,T}$ and sufficiently large $T$,
\begin{align*}
&\left|K_{1,h}\left(u - {t \over T}\right) - K_{1,h}\left(u_{k} - {t \over T}\right)\right|K_{2,h}\left(d\left(x, X_{t,T}\right)\right)\\
&\quad \leq a_{T}K_{h}^{\ast}\left(u_{k} - {t \over T}, d\left(x, X_{t,T}\right)\right)
\end{align*} 
with $K_{h}^{\ast}(v) = K^{\ast}(v/h)$. Define
\begin{align*}
\bar{\psi}_{1}(u,x) &= {1 \over Th\phi(h)}\sum_{t=1}^{T}K_{h}^{\ast}\left(u - {t \over T}, d\left(x, X_{t,T}\right)\right)|W_{t,T}|I(|W_{t,T}| \leq \tau_{T}).
\end{align*}  
Note that $E\left[\left|\bar{\psi}_{1}(u,x)\right|\right]\leq M<\infty$ for some sufficiently large $M$. Then we obtain
\begin{align*}
&\sup_{u \in B_{k,T}}\left|\hat{\psi}_{1}(u,x) - E[\hat{\psi}_{1}(u,x)]\right|\\
&\leq \left|\hat{\psi}_{1}(u_{k},x) - E[\hat{\psi}_{1}(u_{k},x)]\right| + a_{T}\left(\left|\bar{\psi}_{1}(u_{k},x)\right| + E\left[\left|\bar{\psi}_{1}(u_{k},x)\right|\right]\right)\\
&\leq \left|\hat{\psi}_{1}(u_{k},x) - E[\hat{\psi}_{1}(u_{k},x)]\right| + \left|\bar{\psi}_{1}(u_{k},x) - E[\bar{\psi}_{1}(u_{k},x)]\right| + 2Ma_{T}. 
\end{align*}
Hence we have  
\begin{align*}
&P\left(\sup_{u \in B}\left|\hat{\psi}_{1}(u,x) - E[\hat{\psi}_{1}(u,x)]\right|>4Ma_{T}\right)\\
&\leq N\max_{1 \leq k \leq N}P\left(\sup_{u \in B_{k,T}}\left|\hat{\psi}_{1}(u,x) - E[\hat{\psi}_{1}(u,x)]\right|>4Ma_{T}\right) \leq Q_{1,T} + Q_{2,T} 
\end{align*}
where 
\begin{align*}
Q_{1,T} &= N\max_{1 \leq k \leq N}P\left(\left|\hat{\psi}_{1}(u_{k},x) - E[\hat{\psi}_{1}(u_{k},x)]\right|>Ma_{T}\right),\\
Q_{2,T} &= N\max_{1 \leq k \leq N}P\left(\left|\bar{\psi}_{1}(u_{k},x) - E[\bar{\psi}_{1}(u_{k},x)]\right|>Ma_{T}\right).
\end{align*}
We focus on the analysis of $Q_{1,T}$ since $Q_{2,T}$ can be analyzed in almost the same way. Define
\begin{align*}
Z_{t,T}(u,x) &= K_{1,h}\left(u - {t \over T}\right)\left\{ K_{2,h}\left(d(x,X_{t,T})\right)W_{t,T}I(|W_{t,T}|\leq \tau_{T}) \right. \\
&\left. \quad  - E\left[K_{2,h}\left(d(x,X_{t,T})\right)W_{t,T}I(|W_{t,T}|\leq \tau_{T})\right]\right\}.
\end{align*}
Note that the array $\{Z_{t,T}(u,x)\}$ is $\alpha$-mixing  for each fixed $(u,x)$ with mixing coefficients $\alpha_{Z,T}$ such that $\alpha_{Z,T}(k) \leq \alpha(k)$. We apply Lemma \ref{lem: exp-ineq-mixing} below with $\varepsilon=M\alpha_{T}Th\phi(h)$, $b_{T}=C\tau_{T}$ for sufficiently large $C>0$ and $S_{T}=a_{T}^{-1}\tau_{T}^{-1}$. Furthermore, a straightforward extension of Theorem 2 in \cite{Ma05} yields that $\sigma^{2}_{S_{T},T} \leq C'S_{T}h\phi(h)$ with a constant $C'$ independent of $(u,x)$. Note that we can take $M>0$ sufficiently large such that $C' <M$. Therefore, for any fixed $(u,x)$ and sufficiently large $T$,  we have 
\begin{align*}
P\left(\left|\sum_{t=1}^{T}Z_{t,T}(u,x)\right| \geq Ma_{T}Th\phi(h) \right) &\leq  4\exp\left(-{\varepsilon^{2} \over 64\sigma^{2}_{S_{T},T}{T \over S_{T}} + {8 \over 3}\varepsilon b_{T}S_{T}}\right) + 4{T \over S_{T}}\alpha(S_{T})\\
&\leq 4\exp\left(-{M^{2}\log T} \over 64C' + {8 \over 3}CM\right) + 4{T \over S_{T}}AS_{T}^{-\gamma}\\
&\lesssim \exp\left(-{M\log T} \over 64{C'\over M} + {8 \over 3}C\right) + TS_{T}^{-\gamma-1}\\
&\leq \exp\left(-{M\log T} \over 64 + {8 \over 3}C\right) + TS_{T}^{-\gamma-1}\\
&= T^{-{M \over 64 + 3C}} + Ta_{T}^{\gamma+1}\tau_{T}^{\gamma+1}.
\end{align*}
Observe that 
\begin{align*}
R_{1,T}=h^{-1}a_{T}^{-1}T^{-{M \over 64+3C}} &= o(1)\ \text{(for sufficiently large $M>0$)},\\
R_{2,T}=h^{-1}a_{T}^{-1}Ta_{T}^{\gamma+1}\tau_{T}^{\gamma+1} &= h^{-1}T\left(\sqrt{\log T \over Th\phi(h)}\right)^{\gamma \over 2}\rho_{T}^{\gamma + 1}T^{{\gamma +1} \over \zeta}\\
&= {(\log T)^{{\gamma \over 2} + \zeta_{0}(\gamma +1)} \over T^{{\gamma \over 2}-1-{\gamma + 1 \over \zeta}}h^{{\gamma \over 2}+1}\phi(h)^{{\gamma \over 2}}} = o(1)
\end{align*}
Therefore, we have $Q_{1,T} \lesssim O(R_{1,T}) + O(R_{2,T}) = o(1)$. 
\end{proof}

\begin{proof}[Proof of Theorem \ref{Thm: unif-rate}]
Recall that 
\begin{align*}
\hat{m}(u, x) - m(u,x) &= {1 \over \hat{m}_{1}(u, x)}\left(\hat{g}_{1}(u,x) + \hat{g}_{2}(u,x) - m(u,x)\hat{m}_{1}(u,x)\right),
\end{align*}
where
\begin{align*}
\hat{m}_{1}(u,x) &= {1 \over Th\phi(h)}\sum_{t=1}^{T}K_{1,h}\left(u - {t \over T}\right)K_{2,h}\left(d\left(x, X_{t,T}\right)\right),\\
\hat{g}_{1}(u,x) &= {1 \over Th\phi(h)}\sum_{t=1}^{T}K_{1,h}\left(u - {t \over T}\right)K_{2,h}\left(d\left(x, X_{t,T}\right)\right)\varepsilon_{t, T},\\
\hat{g}_{2}(u,x) &= {1 \over Th\phi(h)}\sum_{t=1}^{T}K_{1,h}\left(u - {t \over T}\right)K_{2,h}\left(d\left(x, X_{t,T}\right)\right)m\left({t \over T} , X_{t,T}\right). 
\end{align*}

(Step1) First we give a sketch of the proof. In Steps 1 and 2, we show the following four results: 
\begin{itemize}
\item[(i)] $\sup_{u \in [0,1]}\left|\hat{g}_{1}(u,x)\right| = O_{p}\left(\sqrt{(\log T)/Th\phi(h)}\right)$. 
\item[(ii)] 
\begin{align*}
&\sup_{u \in [0,1]}\left|\hat{g}_{2}(u,x) - m(u,x)\hat{m}_{1}(u,x) - E\left[\hat{g}_{2}(u,x) - m(u,x)\hat{m}_{1}(u,x)\right]\right|\\
&\quad = O_{p}\left(\sqrt{(\log T)/Th\phi(h)}\right). 
\end{align*}
\item[(iii)] $\sup_{u \in [C_{1}h, 1-C_{1}h]}\left|E\left[\hat{g}_{2}(u,x) - m(u,x)\hat{m}_{1}(u,x)\right]\right| = O(h^{2}) + O(h^{\beta})$.
\item[(iv)] $1/\inf_{u \in [C_{1}h, 1-C_{1}h]}\hat{m}_{1}(u,x) = O_{p}(1)$.
\end{itemize}
(i) can be shown by applying Proposition \ref{Prop: unif-rate} with $W_{t,T} = \varepsilon_{t,T}$. (ii) can be shown by applying Proposition \ref{Prop: unif-rate} to $\hat{g}_{2}(u,x) - m(u,x)\hat{m}_{1}(u,x)$. For the proof of (iv), we decompose $\hat{m}_{1}(u,x)$ as follows: 
\begin{align*}
\hat{m}_{1}(u,x) = \tilde{m}_{1}(u,x) + \bar{m}_{1}(u,x),
\end{align*} 
where 
\begin{align*}
\tilde{m}_{1}(u,x) &= {1 \over Th\phi(h)}\sum_{t=1}^{T}K_{1,h}\left(u - {t \over T}\right)K_{2,h}\left(d\left(x, X_{t}^{(t/T)}\right)\right),\\
\bar{m}_{1}(u,x) &= {1 \over Th\phi(h)}\sum_{t=1}^{T}K_{1,h}\left(u - {t \over T}\right)\left\{K_{2,h}\left(d\left(x, X_{t,T}\right)\right)-K_{2,h}\left(d\left(x, X_{t}^{(t/T)}\right)\right)\right\}.
\end{align*}
Applying Proposition \ref{Prop: unif-rate} with $W_{t,T} = 1$, we have that $\sup_{u \in [0,1]}|\hat{m}_{1}(u,x) - E[\hat{m}_{1}(u,x)]|=o_{p}(1)$ uniformly in $u$. Moreover,
\begin{align*}
E[|\bar{m}_{1}(u,x)|] &\lesssim {1 \over Th\phi(h)}\sum_{t=1}^{T}K_{1,h}\left(u-{t \over T}\right){1 \over Th}E[U_{t,T}^{(t/T)}]\\
&\lesssim {o(\phi(h)) \over Th\phi(h)}\sum_{t=1}^{T}K_{1,h}\left(u-{t \over T}\right) = o(1)\ (\text{from Lemma \ref{lem: g}})
\end{align*}
uniformly in $u$. Then we have
\begin{align*}
\hat{m}_{1}(u,x) &= \hat{m}_{1}(u,x) - E[\hat{m}_{1}(u,x)] + E[\hat{m}_{1}(u,x)]\\
&= o_{p}(1) + E[\tilde{m}_{1}(u,x)] + E[\bar{m}_{1}(u,x)]\\
&= E[\tilde{m}_{1}(u,x)] + o_{p}(1) + o(1)
\end{align*}
uniformly in $u$. Observe that
\begin{align*}
E[\tilde{m}_{1}(u,x)] &= {1 \over Th\phi(h)}\sum_{t=1}^{T}K_{1,h}\left(u - {t \over T}\right)E\left[K_{2,h}\left(d\left(x, X_{t}^{(t/T)}\right)\right)\right]\\
&= {1 \over Th\phi(h)}\sum_{t=1}^{T}K_{1,h}\left(u - {t \over T}\right)\int_{0}^{h}K_{2,h}(y)dF_{t/T}(y;x)\\
&\gtrsim {1 \over Th\phi(h)}\sum_{t=1}^{T}K_{1,h}\left(u - {t \over T}\right)\phi(h)f_{1}(x) \sim f_{1}(x)>0
\end{align*}
uniformly in $u$. Note that the last inequality can be obtained by slightly extending Lemma 4.4 in \cite{FeVi06}. Therefore, we obtain 
\begin{align*}
{1 \over \inf_{u \in [C_{1}h, 1-C_{1}h]}\hat{m}_{1}(u,x)} &= {1 \over \inf_{u \in [C_{1}h, 1-C_{1}h]}E[\tilde{m}_{1}(u,x)] + o_{p}(1)+o(1)}=O_{p}(1).
\end{align*}

Combining the results (i), (ii), (iii) and (iv), we have that
\begin{align*}
&\sup_{u \in [C_{1}h, 1-C_{1}h]}\left|\hat{m}(u,x) - m(u,x)\right|\\
&\leq {1 \over \inf_{u \in [C_{1}h, 1-C_{1}h]}\hat{m}_{1}(u,x)}\left(\sup_{u \in [C_{1}h, 1-C_{1}h]}\left|\hat{g}_{1}(u,x)\right| + \sup_{u \in [C_{1}h, 1-C_{1}h]}\left|\hat{g}_{2}(u,x) - m(u,x)\hat{m}_{1}(u,x)\right|\right)\\
&\leq {1 \over \inf_{u \in [C_{1}h, 1-C_{1}h]}\hat{m}_{1}(u, x)}O_{p}\left(\sqrt{\log T \over Th\phi(h)} + h^{2} + h^{\beta}\right) = O_{p}\left(\sqrt{\log T \over Th\phi(h)} + h^2 + h^{\beta}\right).
\end{align*}
Therefore, we complete the proof. 

(Step2) In this step, we show (iii). Let $K_{0}: [0,1] \to \mathbb{R}$ be a Lipschitz continuous function with support $[0,q]$ for some $q>1$. Assume that $K_{0}(x) = 1$ for all $x \in [0,1]$ and write $K_{0,h}(x) = K_{0}(x/h)$. Observe that 
\begin{align*}
E\left[\hat{g}_{2}(u,x) - m(u,x)\hat{m}_{1}(u,x))\right] &= \sum_{i=1}^{4}Q_{i}(u,x),
\end{align*}
where
\begin{align*}
Q_{i}(u,x) &= {1 \over nh\phi(h)}\sum_{t=1}^{T}K_{1,h}\left(u - {t \over T}\right)q_{i}(u,x)
\end{align*}
and 
\begin{align*}
q_{1}(u,x) &= E\left[ K_{0,h}(d\left(x, X_{t,T}\right))\left\{K_{2,h}(d\left(x, X_{t,T}\right)) - K_{2,h}\left(d\left(x, X_{t}^{(t/T)}\right)\right) \right\} \right. \\
&\left .\quad \quad \times \left\{m\left({t \over T}, X_{t,T}\right) - m(u,x)\right\} \right],
\end{align*}
\begin{align*}
q_{2}(u,x) &= E\left[K_{0,h}(d\left(x, X_{t,T}\right))K_{2,h}\left(d\left(x, X_{t}^{(t/T)}\right)\right) \left\{m\left({t \over T}, X_{t,T}\right) - m\left({t \over T},X_{t}^{(t/T)}\right)\right\} \right],
\end{align*}
\begin{align*}
q_{3}(u,x) &= E\left[\left\{K_{0,h}(d\left(x, X_{t,T}\right)) - K_{0,h}\left(d\left(x, X_{t}^{(t/T)}\right)\right) \right\} \right. \\
&\left. \quad \quad K_{2,h}\left(d\left(x, X_{t}^{(t/T)}\right)\right) \left\{m\left({t \over T}, X_{t}^{(t/T)}\right) - m(u,x)\right\}\right],
\end{align*}
\begin{align*}
q_{4}(u,x) &= E\left[K_{2,h}\left(d\left(x, X_{t}^{(t/T)}\right)\right)\left\{m\left({t \over T}, X_{t}^{(t/T)}\right) - m(u,x)\right\} \right].
\end{align*}
We first consider $Q_{1}(u,x)$. Observe that
\begin{align*}
Q_{1}(u,x) &\lesssim {1 \over Th\phi(h)}\sum_{t=1}^{T}K_{1,h}\left(u - {t \over T}\right)E\left[\left|K_{2,h}(d\left(x, X_{t,T}\right)) - K_{2,h}\left(d\left(x, X_{t}^{(t/T)}\right)\right)\right| \right. \\
&\left. \quad \times K_{0,h}(d\left(x, X_{t,T}\right)) \left|m\left({t \over T}, X_{t,T}\right) - m(u,x)\right|\right].
\end{align*}
Note that $K_{0,h}(d\left(x, X_{t,T}\right)) \left|m\left({t \over T}, X_{t,T}\right) - m(u,x)\right| \lesssim h^{1 \wedge \beta}$. Since $K_{2}$ is Lipschitz and $d\left(X_{t,T}, X_{t}^{(t/T)}\right)\leq {1 \over T}U_{t,T}^{(t/T)}$, we have that 
\begin{align*}
&Q_{1}(u,x)\\
&\quad \lesssim {h^{1 \wedge \beta} \over Th\phi(h)}\sum_{t=1}^{T}K_{1,h}\left(u - {t \over T}\right)E\left[\left|K_{2,h}(d\left(x, X_{t,T}\right)) - K_{2,h}\left(d\left(x, X_{t}^{(t/T)}\right)\right)\right|\right]\\
&\quad \lesssim {h^{1\wedge \beta} \over Th\phi(h)}\sum_{t=1}^{T}K_{1,h}\left(u - {t \over T}\right)E\left[\left|{1 \over Th}U_{t,T}^{(t/T)}\right|\right] \lesssim {1 \over Th^{1-(1\wedge \beta)}\phi(h)}
\end{align*}
uniformly in $u$. Using similar arguments, we can also show that 
\begin{align*}
\sup_{u \in [C_{1}h,1-C_{1}h]}|Q_{2}(u,x)| &\lesssim {1 \over Th^{1-(1\wedge \beta)}\phi(h)},\ \sup_{u \in [C_{1}h,1-C_{1}h]}|Q_{3}(u,x)| \lesssim {1 \over Th^{1-(1\wedge \beta)}\phi(h)}. 
\end{align*}
Finally, applying Lemma \ref{lem: K} below and using the assumptions on the smoothness of $m$, we have that $\sup_{u \in [C_{1}h,1-C_{1}h]}|Q_{4}(u,x)| \lesssim h^2 + h^{\beta}$. 
\end{proof}

\begin{proof}[Proof of Theorem \ref{Thm: normality}]
Recall that
\begin{align*}
\hat{m}(u,x) - m(u,x) = {1 \over \hat{m}_{1}(u,x)}\left(\hat{g}_{1}(u,x) + \hat{g}_{2}(u,x) - m(u,x)\hat{m}_{1}(u,x)\right).
\end{align*}
Define $\hat{g}^{B}(u,x) = \left(\hat{g}_{2}(u,x) - m(u,x)\hat{m}_{1}(u,x))\right)$.

(Step1) First, we will show that
\begin{align}\label{bias-neg}
\hat{g}^{B}(u,x) - E[\hat{g}^{B}(u,x)] &= o_{p}\left(\sqrt{{1 \over Th\phi(h)}}\right). 
\end{align}
Define $\Delta_{t,T}(u,x)=K_{2,h}\left(d\left(x,X_{t,T}\right)\right)\left(m\left({t \over T}, X_{t,T}\right) - m(u,x)\right)$. 
Observe that
\begin{align*}
&\Var(\hat{g}^{B}(u,x))\\
&\quad = {1 \over (Th\phi(h))^{2}}\left\{\sum_{t=1}^{T}K_{1,h}^{2}\left(u-{t \over T}\right)\Var(\Delta_{t,T}(u,x)) \right. \\
&\left. \quad + \sum_{t_{1}, t_{2}=1, t_{1} \neq t_{2}}^{T}K_{1,h}\left(u-{t_{1} \over T}\right)K_{1,h}\left(u-{t_{2} \over T}\right)\Cov(\Delta_{t_{1},T}(u,x),\Delta_{t_{2},T}(u,x))\right\}\\
&\quad =: V^{B}_{1,T} + V^{B}_{2,T}. 
\end{align*}
For $V^{B}_{1,T}$, 
\begin{align*}
|V^{B}_{1,T}| &\lesssim {h^{2(1\wedge \beta)} \over (Th\phi(h))^{2}}\sum_{t=1}^{T}K_{1,h}^{2}\left(u-{t \over T}\right)E\left[K_{2,h}^{2}\left(d\left(x,X_{t,T}\right)\right)\right]\\
&\lesssim {h^{2(1\wedge \beta)} \over (Th\phi(h))^{2}}\sum_{t=1}^{T}K_{1,h}^{2}\left(u-{t \over T}\right)\left\{E\left[K_{2,h}^{2}\left(d\left(x,X_{t}^{(t/T)}\right)\right)\right] + {1 \over Th}E[U_{t,T}^{(t/T)}]\right\}\\
&\lesssim {h^{2(1\wedge \beta)}\phi(h) \over (Th\phi(h))^{2}}\sum_{t=1}^{T}K_{1,h}^{2}\left(u-{t \over T}\right) \lesssim {h^{2(1\wedge \beta)}\phi(h) \over (Th\phi(h))^{2}} \ll {1 \over Th\phi(h)}.
\end{align*}

For $V^{B}_{2,T}$, 
\begin{align*}
V^{B}_{2,T} &= {1 \over (Th\phi(h))^{2}} \sum_{t_{1}, t_{2}=1, 1\leq |t_{1} - t_{2}| \leq \lambda_{T}}^{T}K_{1,h}\left(u-{t_{1} \over T}\right)K_{1,h}\left(u-{t_{2} \over T}\right)\Cov(\Delta_{t_{1},T}(u,x),\Delta_{t_{2},T}(u,x))\\
&\quad +{1 \over (Th\phi(h))^{2}} \sum_{t_{1}, t_{2}=1, 1\leq |t_{1} - t_{2}| > \lambda_{T}}^{T}K_{1,h}\left(u-{t_{1} \over T}\right)K_{1,h}\left(u-{t_{2} \over T}\right)\Cov(\Delta_{t_{1},T}(u,x),\Delta_{t_{2},T}(u,x))\\
&=: V^{B}_{21,T} + V^{B}_{22,T} 
\end{align*}
where $\lambda_{T} = o(T)$ at a rate specified in the sequel. For $V^{B}_{21,T}$, 
\begin{align}
|V^{B}_{21,T}| &\leq {1 \over (Th\phi(h))^{2}} \sum_{t_{1}, t_{2}=1, 1\leq |t_{1} - t_{2}| \leq \lambda_{T}}^{T}K_{1,h}\left(u-{t_{1} \over T}\right)K_{1,h}\left(u-{t_{2} \over T}\right) \nonumber \\
&\quad \times \left(E\left[\Delta_{t_{1},T}(u,x)\Delta_{t_{2},T}(u,x)\right] + E\left[\Delta_{t_{1},T}(u,x)\right]E\left[\Delta_{t_{2},T}(u,x)\right]\right)\nonumber \\
&\lesssim {h^{2(1 \wedge \beta)} \over (Th\phi(h))^{2}} \sum_{t_{1}, t_{2}=1, 1\leq |t_{1} - t_{2}| \leq \lambda_{T}}^{T}K_{1,h}\left(u-{t_{1} \over T}\right)K_{1,h}\left(u-{t_{2} \over T}\right)(\psi(h) + \phi^{2}(h)) \nonumber  \\
&\lesssim {h^{2(1 \wedge \beta)}(\psi(h)+\phi^{2}(h)) \over (Th\phi(h))^{2}}T\lambda_{T} \lesssim {1 \over Th\phi(h)}\times {h^{2(1\wedge \beta)-1}\phi(h)\lambda_{T}}. \label{V21-bound}
\end{align}
We shall subsequently select $\lambda_{T}$ to make the right hand side of (\ref{V21-bound}) tends to zero as $T \to \infty$. By Davydov's Lemma (\cite{HaHe80}, Corollary A.2), 
\begin{align*}
&\Cov(\Delta_{t_{1},T}(u,x),\Delta_{t_{2},T}(u,x))\\ 
&\quad \lesssim E[\Delta_{t_{1},T}(u,x)^{\nu}]^{1/\nu}E[\Delta_{t_{2},T}(u,x)^{\nu}]^{1/\nu}(\alpha(|t_{1}-t_{2}|))^{1-{2 \over \nu}}\\
&\quad \lesssim h^{2(1\wedge \beta)}E[K_{2,h}\left(d\left(x,X_{t_{1},T}\right)\right)^{\nu}]^{1/\nu}E[K_{2,h}\left(d\left(x,X_{t_{2},T}\right)\right)^{\nu}]^{1/\nu}(\alpha(|t_{1}-t_{2}|))^{1-{2 \over \nu}}\\
&\quad \lesssim h^{2(1\wedge \beta)}E[K_{2,h}\left(d\left(x,X_{t_{1},T}\right)\right)^{2}]^{1/\nu}E[K_{2,h}\left(d\left(x,X_{t_{2},T}\right)\right)^{2}]^{1/\nu}\alpha(|t_{1}-t_{2}|)^{1-{2 \over \nu}}\\
&\quad \lesssim h^{2(1\wedge \beta)}\phi^{2/\nu}(h)(\alpha(k))^{1-{2 \over \nu}}.
\end{align*}
For the third inequality, we used the boundedness of $K_{2}$. Then for $V^{B}_{22}$, 
\begin{align*}
|V^{B}_{22}| &\lesssim {h^{2(1 \wedge \beta)}\phi^{2/\nu}(h) \over (Th\phi(h))^{2}}\sum_{t_{1}, t_{2}=1, 1\leq |t_{1} - t_{2}| > \lambda_{T}}^{T}(\alpha(|t_{1}-t_{2}|))^{1-{2 \over \nu}}\\
&\lesssim {1 \over Th\phi(h)} \times {h^{2(1\wedge \beta)-1} \over \lambda_{T}^{\delta}(\phi(h))^{1-{2 \over \nu}}}\sum_{k=\lambda_{T}+1}^{\infty}k^{\delta}(\alpha(k))^{1-{2 \over \nu}}.
\end{align*}
Now we select $\lambda_{T}$ as $\lambda_{T} = \lfloor (\phi(h))^{-(1-{2 \over \nu})/\delta}\rfloor$. Then by Assumption \ref{Ass-E}, 
\begin{align*}
\Var(\hat{g}^{B}(u,x)) &\leq |V^{B}_{1,T}| + |V^{B}_{2,T}| = o\left({1 \over Th\phi(h)}\right). 
\end{align*}
This yields (\ref{bias-neg}). From the argument in (Step1) of the proof of Theorem \ref{Thm: unif-rate}, we have $E[\hat{g}^{B}(u,x)]=O(h^{2\wedge \beta})$, $\hat{m}_{1}(u,x) = E[\hat{m}_{1}(u,x)] + o_{p}(1)$ and $\lim_{T \to \infty}E[\hat{m}_{1}(u,x)]>0$. Therefore, 
\begin{align*}
\hat{m}(u,x)-m(u,x) &= {\hat{g}_{1}(u,x) \over \hat{m}_{1}(u,x)}  + B_{T}(u,x) + o_{p}\left({1 \over Th\phi(h)}\right).
\end{align*}

(Step2) In this step, we will show 
\begin{align*}
Th\phi(h)\Var(\hat{g}_{1}(u,x)) &\sim E[\varepsilon^{2}_{1}]\sigma^{2}(u,x)\int K^{2}_{1}(w)dw>0\ \text{as $T \to \infty$}.
\end{align*}
Define $\tilde{g}_{1}(u,x)=\sqrt{Th\phi(h)}\hat{g}_{1}(u,x)$. Observe that
\begin{align*}
\Var(\tilde{g}_{1}(u,x)) &= {1 \over Th\phi(h)}\sum_{t=1}^{T}K_{1,h}^{2}\left(u-{t \over T}\right)E\left[K_{2,h}^{2}\left(d\left(x,X_{t,T}\right)\right)\varepsilon_{t,T}^{2}\right]\\
&={\sigma^{2}(u,x)+o(1) \over Th\phi(h)}\sum_{t=1}^{T}K_{1,h}^{2}\left(u-{t \over T}\right)E\left[K_{2,h}^{2}\left(d\left(x,X_{t,T}\right)\right)\varepsilon_{t}^{2}\right]\\
&={E[\varepsilon_{1}^{2}](\sigma^{2}(u,x)+o(1)) \over Th\phi(h)}\sum_{t=1}^{T}K_{1,h}^{2}\left(u-{t \over T}\right)E\left[K_{2,h}^{2}\left(d\left(x,X_{t,T}\right)\right)\right].
\end{align*}
Since 
\begin{align*}
E\left[\left|K_{2,h}^{2}\left(d\left(x,X_{t,T}\right)\right)-K_{2,h}^{2}\left(d\left(x,X_{t}^{(t/T)}\right)\right)\right|\right] &\lesssim E\left[\left|K_{2,h}\left(d\left(x,X_{t,T}\right)\right)-K_{2,h}\left(d\left(x,X_{t}^{(t/T)}\right)\right)\right|\right]\\
&\lesssim {1 \over Th}E[U_{t,T}^{(t/T)}] = o(\phi(h)),
\end{align*}
we have
\begin{align*}
\Var(\tilde{g}_{1}(u,x)) &={E[\varepsilon_{1}^{2}](\sigma^{2}(u,x)+o(1)) \over Th\phi(h)}\sum_{t=1}^{T}K_{1,h}^{2}\left(u-{t \over T}\right)E\left[K_{2,h}^{2}\left(d\left(x,X_{t}^{(t/T)}\right)\right)\right]\\
&\quad + {E[\varepsilon_{1}^{2}]o(\phi(h))(\sigma^{2}(u,x)+o(1)) \over Th\phi(h)}\sum_{t=1}^{T}K_{1,h}^{2}\left(u-{t \over T}\right)\\
&= {E[\varepsilon_{1}^{2}](\sigma^{2}(u,x)+o(1)) \over Th\phi(h)}\sum_{t=1}^{T}K_{1,h}^{2}\left(u-{t \over T}\right)E\left[K_{2,h}^{2}\left(d\left(x,X_{t}^{(t/T)}\right)\right)\right] + o(1). 
\end{align*}
By the integration by parts and change of variables, 
\begin{align*}
E\left[K_{2,h}^{2}\left(d\left(x,X_{t}^{(t/T)}\right)\right)\right] &= -{2 \over h}\int_{0}^{h}K_{2,h}(y)K'_{2,h}(y)F_{t/T}(y;x)dy\\
&\sim -{2 \over h}\int_{0}^{h}K_{2,h}(y)K'_{2,h}(y)\phi(y)dy\\
&= {2 \over h}\int_{0}^{h}\left(1-{y\over h}\right)\phi(y)dy\\
&= {2 \over h^2}\int_{0}^{h}\left(\int_{0}^{y}\phi(z)dz\right)dy\\
&\sim {2 \over h^{2}}\int_{0}^{h}y\phi(y)dy \sim {1 \over h^{2}}h^{2}\phi(h) \sim \phi(h).
\end{align*}
Threfore, we have
\begin{align*}
\Var(\tilde{g}_{1}(u,x)) &\sim {E[\varepsilon_{1}^{2}](\sigma^{2}(u,x)+o(1)) \over Th}\sum_{t=1}^{T}K_{1,h}^{2}\left(u-{t \over T}\right)\\
&\sim E[\varepsilon_{1}^{2}]\sigma^{2}(u,x)\int K_{1}^{2}(w)dw. 
\end{align*}
(Step 3) Moreover, $\tilde{g}_{1}(u,x)$ is asymptotically normal. In particular, 
\begin{align}\label{asy-normality}
\tilde{g}_{1}(u,x) \stackrel{d}{\to} N(0, V(u,x))\ \text{as $T \to \infty$}.
\end{align}
We can show (\ref{asy-normality}) by applying blocking arguments of \cite{Be26} and Volkonskii and Rozanov inequality (cf. Proposition 2.6 in \cite{FaYa03}). Assumption \ref{Ass-block} implies that there exists a sequence of positive integers $\{q_T\}$ such that as $T \to \infty$, $q_T \to \infty$, 
\[
q_Tv_T = o(\sqrt{Th\phi(h)}),\ q_T\sqrt{T \over h\phi(h)}\alpha(v_T) \to 0.
\] 
Decompose $\tilde{g}_{1}(u,x)$ into  big-blocks and small-blocks as follows:
\begin{align*}
\tilde{g}_1(u,x) &= {1 \over \sqrt{Th\phi(h)}}\sum_{j=1}^{k_T}\xi_{j}(u,x) +  {1 \over \sqrt{Th\phi(h)}}\sum_{j=1}^{k_T}\eta_{j}(u,x) + \zeta(u,x)\\
&=: \tilde{g}_{11}(u,x) + \tilde{g}_{12}(u,x) + \tilde{g}_{13}(u,x),
\end{align*} 
where 
\begin{align*}
\xi_{j}(u,x) &= \sum_{t = (j-1)(\ell_T + s_T) + 1}^{j\ell_T + (j-1)s_T}K_{1,h}(u-t/T)K_{2,h}(d(x,X_{t,T}))\varepsilon_{t,T},\\
\eta_{j}(u,x) &= \sum_{t = j\ell_T + (j-1)s_T + 1}^{j(\ell_T + s_T)}K_{1,h}(u-t/T)K_{2,h}(d(x,X_{t,T}))\varepsilon_{t,T},\\
\zeta(u,x) &= \sum_{t = k_T(\ell_T + s_T) + 1}^{T}K_{1,h}(u-t/T)K_{2,h}(d(x,X_{t,T}))\varepsilon_{t,T},
\end{align*}
and where $\ell_T = \lfloor (Th\phi(h))^{1/2}/q_T\rfloor$, $s_T = v_T$, $k_T = \lfloor T/(\ell_T + s_T)\rfloor$. 
We can neglect the sum of small blocks $\tilde{g}_{12}(u,x)$ and $\tilde{g}_{13}(u,x)$, and exploit the mixing conditions to replace the big
blocks $\xi_{j}(u,x)$ by independent random variables. This allows us to apply a Lindeberg theorem to get the result. We omit the details as the proof is similar to that of Theorem 4 in \cite{Ma05}. Combining (\ref{asy-normality}) and the results in Steps 1 and 2, we obtain the conclusion.  
\end{proof}

\section{Technical Tools}

In this section we provide some lemmas used in the proofs of main results. The proofs of following Lemmas \ref{lem: K} and \ref{lem: g} are straightforward and thus omitted. Let $I_{h}=[C_{1}h, 1-C_{1}h]$.

\begin{lemma}\label{lem: K}
Suppose that kernel $K_{1}$ satisfies Assumption \ref{Ass-KB} (KB1). Then for $k=0,1,2$, 
\begin{align*}
\sup_{u \in I_{h}}\left|{1 \over Th}\sum_{t=1}^{T}K_{1,h}\left(u-{t \over T}\right)\left({u - t/T \over h}\right)^{k} - \int_{0}^{1}{1 \over h}K_{1,h}(u-v)\left({u - v \over h}\right)^{k}dv\right| = O\left({1 \over Th^{2}}\right). 
\end{align*}
\end{lemma}

\begin{lemma}\label{lem: g}
Suppose that kernel $K_{1}$ satisfies Assumption \ref{Ass-KB} (KB1) and let $g: [0,1] \times \mathscr{H} \to \mathbb{R}$, $(u,x) \mapsto g(u,x)$ be continuously differentiable with respect to $u$. Then, 
\begin{align*}
\sup_{u \in I_{h}}\left|{1 \over Th}\sum_{t=1}^{T}K_{1,h}\left(u-{t \over T}\right)g\left({t \over T}, x\right) - g(u,x)\right| = O\left({1 \over Th^{2}}\right) + o(h). 
\end{align*}
\end{lemma}

The following result is an exponential inequality for strongly mixing sequences given in \cite{Li96}. 

\begin{lemma}[Theorem 2.1 in \cite{Li96}]\label{lem: exp-ineq-mixing}
Let $\{Z_{t,T}\}$ be a zero-mean triangular array such that $|Z_{t,T}| \leq b_{T}$ with $\alpha$-mixing coefficients $\alpha(k)$. Then for any $\varepsilon>0$ and $S_{T}\leq T$ with $\varepsilon>4S_{T}b_{T}$, 
\[
P\left(\left|\sum_{t=1}^{T}Z_{t,T}\right| \geq \varepsilon \right) \leq 4\exp\left(-{\varepsilon^{2} \over 64\sigma^{2}_{S_{T},T}{T \over S_{T}} + {8 \over 3}\varepsilon b_{T}S_{T}}\right) + 4{T \over S_{T}}\alpha(S_{T})
\]
where $\sigma^{2}_{S_{T},T} = \sup_{0 \leq j \leq T-1}E\left[\left(\sum_{t=j+1}^{(j+S_{T})\wedge T}Z_{t,T}\right)^{2}\right]$.
\end{lemma}

\end{document}